\documentclass[11pt, a4paper, twoside, leqno]{article}
\usepackage[T1]{fontenc}
\usepackage{lmodern}
\usepackage{amsmath,amssymb,amsthm}
\usepackage{mathtools}
\usepackage[margin=1.05in]{geometry}
\usepackage{microtype}
\usepackage{booktabs}
\usepackage{enumitem}
\usepackage[hidelinks]{hyperref}

\hypersetup{
  pdftitle={Endpoint maximal estimates for the half-wave propagator},
  pdfsubject={A research manuscript on critical Sobolev wave maximal estimates},
  pdfkeywords={wave equation, maximal estimate, Sobolev endpoint, radial frequency decomposition}
}
\numberwithin{equation}{section}
\newtheorem{theorem}{Theorem}[section]
\newtheorem{proposition}[theorem]{Proposition}
\newtheorem{lemma}[theorem]{Lemma}
\newtheorem{corollary}[theorem]{Corollary}
\theoremstyle{remark}
\newtheorem{remark}[theorem]{Remark}
\newcommand{\R}{\mathbb R}

\newcommand{\U}{U}
\newcommand{\M}{\mathcal M}
\newcommand{\A}{\mathbb A}
\newcommand{\norm}[1]{\lVert #1\rVert}

\newcommand{\dd}{\,\mathrm d}

\DeclareMathOperator{\supp}{supp}

\allowdisplaybreaks[1]
\let\oldthanks\thanks
\renewcommand{\thanks}[1]{%
	\renewcommand{\thefootnote}{}%  清除当前脚注标记
	\oldthanks{#1}%                 调用原始的 \thanks
	\renewcommand{\thefootnote}{\arabic{footnote}}% 恢复为数字编号（或您需要的样式）
}

\begin{document}

\title {\textbf{ Endpoint maximal estimates for the half-wave propagator}	
	\author{Bowen An, \hspace{0.2cm}Wenjuan Li, \hspace{0.2cm}Huiju Wang}
	\hspace{0.2cm}
	\thanks{This work is supported by the National Key R\&D Program of China (No.2023YFA1010800); Natural Science Foundation of China (No.12301113; No.12271435).}}

\date{}
\maketitle
\thispagestyle{empty}

\begin{abstract}
We present a proof of the strong maximal estimate for the Euclidean
half-wave propagator at the critical pair
\(s_d=d/(d+1)\) and \(q_d=2(d+1)/(d-1)\), in every dimension \(d\ge2\), i.e.
\begin{equation*}
\biggl\|\sup_{0<t<1}|e^{it\sqrt{-\Delta}}f(x)|\biggl\|_{L^{q_d}(\R^d)}
	\le C_d\norm{f}_{ H^{s_d}(\R^d)}.
\end{equation*}
\end{abstract}

\section{Introduction}
\subsection{Background and main results}
For \(d\ge2\), let
\begin{equation}\label{eq:propagator}
 \U(t)f(x)=e^{it\sqrt{-\Delta}}f(x)
   =(2\pi)^{-d}\int_{\R^d}
      e^{i(x\cdot\xi+t|\xi|)}\widehat f(\xi)\dd\xi.
\end{equation}
The associated maximal operators are
\begin{equation}\label{eq:maximal-def}
 \M_{\R}f(x)=\sup_{t\in\R}|\U(t)f(x)|,
 \qquad
 \M_{(0,1)}f(x)=\sup_{0<t<1}|\U(t)f(x)|.
\end{equation}
These expressions are initially defined for Schwartz functions.

The local-in-time Sobolev maximal inequality is
\begin{equation}\label{eq:local-general}
 \norm{\M_{(0,1)}f}_{L^q(\R^d)}
       \le C\norm{f}_{H^\sigma(\R^d)}.
\end{equation}
Rogers and Villarroya~\cite{RV} identified the sharp regularity
threshold up to equality. For \(2\le q<\infty\), the threshold is
\begin{equation}\label{eq:threshold}
 \sigma_c(q,d)=
 \max\left\{
     \frac d2-\frac dq,\,
     \frac{d+1}{4}-\frac{d-1}{2q}
     \right\}.
\end{equation}
In particular, their results give sufficiency above this threshold
and failure below it.
Cho, Lee, and Li~\cite{CLL} proved the strong estimates at
\(\sigma=\sigma_c(q,d)\) for
\(2<q<\infty\), with the critical exponent
\begin{equation}\label{eq:critical-pair}
 q_d=\frac{2(d+1)}{d-1},
 \qquad
 s_d=\sigma_c(q_d,d)=\frac d{d+1}
\end{equation}
excluded. The case \(q=q_d\) is explicitly left open after
Theorem~1.1 of~\cite{CLL}.
The argument below addresses this remaining pair.

\begin{theorem}\label{thm:main}
Let \(d\ge2\), and let \(q_d,s_d\) be as in
\eqref{eq:critical-pair}. There is a constant \(C_d\) such that
\begin{equation}\label{eq:main}
 \norm{\M_{\R}f}_{L^{q_d}(\R^d)}
       \le C_d\norm{f}_{\dot H^{s_d}(\R^d)}
\end{equation}
for every \(f\in\mathcal S(\R^d)\).
The estimate extends to the homogeneous Sobolev space
\(\dot H^{s_d}(\R^d)\).
\end{theorem}

Since the inhomogeneous Sobolev space $H^{s_d}(\R^d)$
is continuously embedded into the homogeneous Sobolev space $\dot H^{s_d}(\R^d)$, then we directly have the following corollary.
\begin{corollary}\label{cor:local}
Let \(d\ge2\), and let \(q_d,s_d\) be as in
\eqref{eq:critical-pair}. There is a constant \(C_d\) such that. The estimate
\begin{equation}\label{eq:local-endpoint}
 \norm{\M_{(0,1)}f}_{L^{q_d}(\R^d)}
       \le C_{d}\norm{f}_{H^{s_d}(\R^d)}
\end{equation}
holds for $f\in H^{s_d}(\R^d)$.
\end{corollary}

The main ingredient is the following product inequality. Set
\begin{equation}\label{eq:delta}
 p_d=\frac{q_d}{2}=\frac{d+1}{d-1},
 \qquad
 \delta_d=
 \begin{cases}
   1/30,& d=2,\\
   1/(d^2-1),& d\ge3.
 \end{cases}
\end{equation}
For \(N>0\), set
\[
 \A_N=\{\xi\in\R^d:N/2\le|\xi|\le2N\}.
\]

\begin{proposition}\label{prop:bilinear-intro}
Suppose \(N\ge M>0\),
\(\supp\widehat f_N\subset\A_N\), and
\(\supp\widehat g_M\subset\A_M\), with \(f_N,g_M\in L^2(\R^d)\).
Then
\begin{equation}\label{eq:bilinear-intro}
 \left\|\sup_{t\in\R}
       |\U(t)f_N\,\U(t)g_M|\right\|_{L^{p_d}(\mathbb{R}^d)}
 \le C_d\left(\frac MN\right)^{\delta_d}
       \bigl(N^{s_d}\norm{f_N}_{L^2}\bigr)
       \bigl(M^{s_d}\norm{g_M}_{L^2}\bigr).
\end{equation}
\end{proposition}

The frequency ratio in \eqref{eq:bilinear-intro} is essential for the
argument. A single-annulus maximal estimate, followed by an
unweighted sum over annuli, does not supply the Sobolev square sum.
In contrast, the kernel \(2^{-\delta_d|j-k|}\) obtained from
\eqref{eq:bilinear-intro} is summable.

Here is the mechanism in dimensions \(d\ge3\).
Write \(p=p_d\), \(s=s_d\), and choose linear Strichartz regularities
\(\alpha,\beta\) satisfying \(\alpha+\beta=1\).
Divide the higher annulus into \(K\lesssim N/M\) disjoint radial
intervals of width at most \(M\). A product involving one such
interval and the lower annulus has temporal Fourier support in an
interval of length \(O(M)\). Its location may be of order \(N\),
but a scalar modulation removes the center without changing its
absolute value. Temporal Bernstein inequality costs \(M^{1/p}\), while summing
the input \(L^2\) norms costs \(K^{1/2}\). The resulting coefficient is
\begin{equation}\label{eq:mechanism}
 N^{\alpha+1/2}M^{\beta+1/p-1/2}.
\end{equation}
Since \(s=1/2+1/(2p)\), a positive frequency gain follows whenever
\[
 1-\frac p2<\alpha<\frac1{2p}.
\]
This interval is nonempty for every \(d\ge3\), because its length is
\((p-1)^2/(2p)>0\).
All the linear estimates can be chosen with time exponent greater
than two and finite spatial exponent.

\subsection{Organization of the Paper}

Section~\ref{sec:prelim} records the classical estimates and the
temporal support facts used below.
Section~\ref{sec:bilinear} proves
Proposition~\ref{prop:bilinear-intro}.
Section~\ref{sec:main-proof} carries out the frequency summation.

\subsection{AI usage}
We used ChatGPT 6.1 for English-language editing and proofreading, as well as to explore possible approaches to parts of the proofs. The principal ideas underlying the results were suggested by ChatGPT 6.1. The authors developed, carefully verified, and rewrote all the mathematical arguments, simplifying most of the proofs in the process. The authors are responsible for this paper.

\section{Preliminaries}\label{sec:prelim}
\subsection{Notations and Sobolev spaces}
\label{subsec:notation}

We use the Fourier transform convention
\[
\widehat f(\xi)
=\mathcal F_xf(\xi)
=\int_{\mathbb R^d}e^{-ix\cdot\xi}f(x)\,\dd x,
\]
with inverse transform
\[
f(x)
=\mathcal F_x^{-1}\widehat f(x)
=(2\pi)^{-d}\int_{\mathbb R^d}
e^{ix\cdot\xi}\widehat f(\xi)\,\dd\xi.
\]
The same convention is used in the time variable:
\[
\mathcal F_tF(\tau)
=\int_{\mathbb R}e^{-it\tau}F(t)\,\dd t,
\qquad
F(t)
=(2\pi)^{-1}\int_{\mathbb R}
e^{it\tau}\mathcal F_tF(\tau)\,\dd\tau.
\]
We write \(\mathcal S(\mathbb R^n)\) and
\(\mathcal S'(\mathbb R^n)\) for the Schwartz space and the space of
tempered distributions, respectively. Unless stated otherwise, Fourier
support conditions are understood in the sense of tempered
distributions.

For \(s\in\mathbb R\), the inhomogeneous Sobolev space is defined by
\[
H^s(\mathbb R^d)
=\left\{
f\in\mathcal S'(\mathbb R^d):
(1+|\xi|^2)^{s/2}\widehat f(\xi)\in L^2(\mathbb R^d)
\right\},
\]
equipped with the norm
\[
\|f\|_{H^s}^2
=(2\pi)^{-d}\int_{\mathbb R^d}
(1+|\xi|^2)^s|\widehat f(\xi)|^2\,\dd\xi.
\]

For \(0<s<d/2\), we use the canonical realization
\[
\dot H^s(\mathbb R^d)
=
\left\{
f\in\mathcal S'(\mathbb R^d):
\widehat f\in L^1_{\mathrm{loc}}(\mathbb R^d),
\quad
|\xi|^s\widehat f(\xi)\in L^2(\mathbb R^d)
\right\},
\]
where the condition
\(\widehat f\in L^1_{\mathrm{loc}}(\mathbb R^d)\) means that
\(\widehat f\) is represented by a locally integrable function. Its norm
is given by
\[
\|f\|_{\dot H^s}^2
=(2\pi)^{-d}\int_{\mathbb R^d}
|\xi|^{2s}|\widehat f(\xi)|^2\,\dd\xi.
\]

We denote by \(|D|^s\) the Fourier multiplier with symbol
\(|\xi|^s\), namely,
\[
|D|^sf
=\mathcal F_x^{-1}
\bigl(|\xi|^s\widehat f(\xi)\bigr).
\]
Plancherel's theorem gives
\[
\|\U(t)f\|_{\dot H^s}
=\|f\|_{\dot H^s}
=\||D|^sf\|_{L^2},
\qquad
f\in\dot H^s(\mathbb R^d),\quad t\in\mathbb R.
\]

If \(h\in L^2(\mathbb R^d)\) has compact Fourier support, then
\(\widehat h\in L^1(\mathbb R^d)\). In this case, we always take
\(\U(t)h\) to be the representative defined by the absolutely
convergent Fourier integral
\[
\U(t)h(x)
=(2\pi)^{-d}\int_{\mathbb R^d}
e^{i(x\cdot\xi+t|\xi|)}
\widehat h(\xi)\,\dd\xi.
\]

Mixed norms are evaluated from right to left. For example,
\[
\|F\|_{L_t^pL_x^r}
=
\left(
\int_{\mathbb R}
\left(
\int_{\mathbb R^d}|F(x,t)|^r\,\dd x
\right)^{p/r}
\dd t
\right)^{1/p},
\]
with the usual modifications when either exponent is infinite. For a
time interval \(J\subset\mathbb R\), we write
\[
\|F\|_{L_x^qL_t^\infty(\mathbb R^d\times J)}
=
\left\|
\operatorname*{ess\,sup}_{t\in J}|F(\cdot,t)|
\right\|_{L_x^q(\mathbb R^d)}.
\]
Thus, in \(L_x^qL_t^\infty\), the temporal essential supremum is taken
before the spatial \(L^q\) norm. Whenever \(t\mapsto F(x,t)\) is
continuous, the temporal essential supremum agrees with the ordinary
supremum.

We also use the abbreviations
\[
L_{x,t}^p
=L^p(\mathbb R^d\times\mathbb R).
\]
If \(I,J\subset\mathbb R\), their Minkowski sum is
\[
I+J=\{\sigma+\tau:\sigma\in I,\ \tau\in J\}.
\]

Finally, \(X\lesssim Y\) means that \(X\le CY\) for a constant \(C>0\).
Unless otherwise indicated, \(C\) may depend on the dimension, the
fixed exponents, and the fixed smooth cut-offs, but is independent of
frequency parameters such as \(N\) and \(M\).

\subsection{A temporal Bernstein inequality}
We shall use the following temporal Bernstein inequality, which controls a supremum in time by a space-time \(L^p\) norm. It is a consequence of the one-dimensional Bernstein inequality (see \cite[Lemma~2.1]{BCD}); since modulation removes the center of the temporal frequency interval, the constant depends only on its width and is uniform in its location.
\begin{lemma}\label{lem:time-bernstein}
	Let \(1<p<\infty\), \(L>0\), and \(A\in\mathbb R\). Suppose that
	\(F\in L^p(\mathbb R^d\times\mathbb R)\) admits a representative such that
	\(x\mapsto F(x,t)\) is measurable for every \(t\in\mathbb R\) and,
	for almost every \(x\in\mathbb R^d\), the function
	\(t\mapsto F(x,t)\) is bounded and continuous, with
	\[
	\operatorname{supp}\mathcal F_tF(x,\cdot)\subset[A-L,A+L]
	\]
	in \(\mathcal S'(\mathbb R)\). Then
	\begin{equation}\label{eq:time-bernstein}
		\left\|\sup_{t\in\mathbb R}|F(x,t)|\right\|_{L_x^p}
		\le C_pL^{1/p}\|F\|_{L_{x,t}^p},
	\end{equation}
	where \(C_p\) is independent of \(A\) and \(L\).
\end{lemma}

\begin{proof}
	For almost every \(x\), Fubini's theorem yields \(F(x,\cdot)\in L^p(\mathbb R)\). The function
	\(G_x(t)=e^{-iAt}F(x,t)\) has Fourier support contained in \([-L,L]\).
	The one-dimensional Bernstein inequality therefore yields
	\[
	\|F(x,\cdot)\|_{L_t^\infty}
	=\|G_x\|_{L_t^\infty}
	\le C_pL^{1/p}\|G_x\|_{L_t^p}
	=C_pL^{1/p}\|F(x,\cdot)\|_{L_t^p}.
	\]
	Since \(G_x\) is continuous in \(t\), its essential supremum equals its supremum.
	After setting \(F(x,t)=0\) for \(x\) in the exceptional null set, this supremum is measurable in \(x\). Taking the \(L_x^p\) norm then yields \eqref{eq:time-bernstein}.
\end{proof}
\subsection{Non-endpoint Strichartz estimates on annuli}

We recall the frequency-localized form of the standard Strichartz
estimate for the half-wave propagator; see
\cite[Corollary~1.3 and Section~8]{KT}.

\begin{lemma}\label{lem:strichartz}
	Let \(d\ge2\) and \(2<p,r<\infty\) satisfy
	\begin{equation}\label{eq:wave-admissible}
		\frac1p+\frac{d-1}{2r}=\frac{d-1}{4},
	\end{equation}
	and set
	\[
	\gamma=\frac d2-\frac dr-\frac1p.
	\]
	Then, for every \(N>0\) and every \(h_N\in L^2(\mathbb R^d)\) with
$\operatorname{supp}\widehat h_N\subset\mathcal \A_N$,
	we have
	\begin{equation}\label{eq:strichartz}
		\|\U(t)h_N\|_{L_t^pL_x^r}
		\le C_{d,p,r}N^\gamma\|h_N\|_{L^2}.
	\end{equation}
	The constant is independent of \(N\) and of any further restriction
	of the Fourier support within \(\mathcal \A_N\).
\end{lemma}

The factor \(N^\gamma\) is a consequence of the homogeneous Strichartz estimate together with the Fourier support condition:
\[
\|\U(t)h_N\|_{L_t^pL_x^r}
\lesssim\|h_N\|_{\dot H^\gamma}
\lesssim N^\gamma\|h_N\|_{L^2}.
\]
Since this estimate holds for every datum whose Fourier support lies in \(\mathcal \A_N\), the same constant remains valid for data supported in any subset of this annulus.

\subsection{Temporal Fourier support of products of half-waves}
To apply the temporal Bernstein inequality to products of frequency-localized waves, we first need to identify their Fourier support in the time variable. For the half-wave propagator
$U(t)$,
the temporal frequency is determined by the radial spatial frequency $|\xi|$. The following lemma makes this correspondence precise and shows that multiplication of two waves adds their temporal frequencies.
\begin{lemma}\label{lem:temporal-support}
	Let \(h,g\in L^2(\R^d)\) and suppose that
	\[
	\supp\widehat h\subset\{\xi:|\xi|\in I\},
	\qquad
	\supp\widehat g\subset\{\xi:|\xi|\in J\},
	\]
	where \(I,J\subset[0,\infty)\) are compact intervals.
	Then the Fourier integral representations of \(\U(t)h(x)\) and \(\U(t)g(x)\) are bounded and jointly continuous in $(x,t)$. Moreover, for every $x\in\mathbb R^d$, we have
	\begin{equation}\label{eq:support-sum}
		\operatorname{supp}\mathcal F_t\biggl(\U(t)h(x)\U(t)g(x)\biggr)
		\subset I+J,
	\end{equation}
	where the temporal Fourier transform is understood in the sense of tempered distributions.
\end{lemma}
\begin{proof}
	Since $I$ and $J$ are compact, the Fourier transforms $\hat{h}$ and $\hat{g}$ are supported on the sets of finite measure and hence belong to $L^1(\mathbb R^d)$ by the Cauchy-Schwarz inequality. Their Fourier integral representations are therefore absolutely convergent and uniformly bounded. Since the phases depend  continuously on $(x,t)$, while the integrals are dominated by $|\hat h|$ and $|\hat g|$, the dominated convergence theorem further shows that $U(t)h(x)$ and $U(t)g(x)$ are jointly continuous in $(x,t)$.
	
	We now determine the temporal Fourier support of the product. Fix
	\(x\in\R^d\). To collect the spatial frequencies according to their
	associated temporal frequency \(|\xi|\), define a complex measure
	\(\mu_x\) on \(\R\) by
	\[
	\mu_x(E)
	=(2\pi)^{-d}
	\int_{\{\xi:\,|\xi|\in E\}}
	e^{ix\cdot\xi}\widehat h(\xi)\dd\xi
	\]
	for every Borel set \(E\subset\R\). Since
	\(\widehat h\in L^1(\R^d)\), the measure \(\mu_x\) is finite.
	Moreover, the frequency localization of \(h\) implies that
	\[
	\supp\mu_x\subset I.
	\]
	By construction,
	\[
	\U(t)h(x)
	=\int_{\R}e^{it\sigma}\dd\mu_x(\sigma).
	\]
	Similarly, there is a finite complex measure \(\nu_x\), supported in
	\(J\), such that
	\[
	\U(t)g(x)
	=\int_{\R}e^{it\tau}\dd\nu_x(\tau).
	\]
	
	Since both measures are finite, Fubini's theorem gives
	\[
	\begin{aligned}
		\U(t)h(x)\U(t)g(x)
		&=\int_{\R}\int_{\R}
		e^{it\sigma}e^{it\tau}
		\dd\mu_x(\sigma)\dd\nu_x(\tau)\\
		&=\int_{\R}\int_{\R}
		e^{it(\sigma+\tau)}
		\dd\mu_x(\sigma)\dd\nu_x(\tau)\\
		&=\int_{\R}e^{it\rho}
		\dd(\mu_x*\nu_x)(\rho).
	\end{aligned}
	\]
	Thus, the product has a Fourier--Stieltjes representation associated
	with the convolution measure \(\mu_x*\nu_x\). With our temporal Fourier
	convention,
	\[
	\mathcal F_t\bigl(\U(t)h(x)\U(t)g(x)\bigr)
	=2\pi(\mu_x*\nu_x)
	\]
	in \(\mathcal S'(\R)\). Since
	\[
	\operatorname{supp}(\mu_x*\nu_x)
	\subset \operatorname{supp}\mu_x+\operatorname{supp}\nu_x
	\subset I+J,
	\]
	we obtain that
	\[
	\operatorname{supp}\mathcal F_t\bigl(\U(t)h(x)\U(t)g(x)\bigr)
	\subset I+J,
	\]
	This proves the lemma.
\end{proof}

\section{A maximal product estimate with frequency decay}
\label{sec:bilinear}

\subsection{Choice of exponents in dimensions \texorpdfstring{\(d\ge3\)}{d >= 3}}
We first choose two  Strichartz-admissible pairs so that H\"{o}lder inequality gives an $L_{x,t}^p$ estimate for the product of two half-waves. The corresponding regularities are chosen so that the radial decomposition in the next subsection produces a positive power of the frequency ratio.
\begin{lemma}\label{lem:parameters}
	Let \(d\ge3\), and set
	\[
	p=\frac{d+1}{d-1},
	\qquad
	s=\frac d{d+1}.
	\]
	Define
	\begin{equation}\label{eq:parameters}
		\alpha=\frac12-\frac d{d^2-1},
		\qquad
		\beta=\frac12+\frac d{d^2-1},
		\qquad
		\delta=\frac1{d^2-1}.
	\end{equation}
	For each \(\gamma\in\{\alpha,\beta\}\), let
	\begin{equation}\label{eq:ar}
		a_\gamma=\frac p\gamma,
		\qquad
		\frac1{r_\gamma}=\frac12-\frac{2\gamma}{d+1}.
	\end{equation}
	Then \((a_\gamma,r_\gamma)\) satisfies the hypotheses of
	Lemma~\ref{lem:strichartz} and has Sobolev regularity \(\gamma\).
	Moreover,
	\begin{equation}\label{eq:holder}
		\frac1{a_\alpha}+\frac1{a_\beta}
		=
		\frac1{r_\alpha}+\frac1{r_\beta}
		=
		\frac1p,
	\end{equation}
	and
	\begin{equation}\label{eq:gain-identities}
		\alpha+\frac12=s-\delta,
		\qquad
		\beta+\frac1p-\frac12=s+\delta.
	\end{equation}
\end{lemma}

\begin{proof}
	We first verify the range of the exponents.
	Since \(d^2-2d-1>0\) for \(d\ge3\), it follows that
	\(0<\alpha<\beta<1\). Moreover,
	\[
	p-2\beta=\frac2{d^2-1}>0,
	\]
	and hence
	$2<a_\beta<a_\alpha<\infty$.
	For each \(\gamma\in\{\alpha,\beta\}\), the inequalities
	$0<\gamma<1\le\frac{d+1}{4}$ imply that
	\[
	0<\frac1{r_\gamma}<\frac12.
	\]
	Consequently, \(2<r_\gamma<\infty\).
	
	By the definitions of \(a_\gamma\) and \(r_\gamma\), we have
	\[
	\begin{aligned}
		\frac1{a_\gamma}+\frac{d-1}{2r_\gamma}
		=
		\frac{\gamma(d-1)}{d+1}
		+\frac{d-1}{4}
		-\frac{\gamma(d-1)}{d+1}=\frac{d-1}{4}.
	\end{aligned}
	\]
	Therefore both pairs satisfy the admissibility condition in
	Lemma~\ref{lem:strichartz}. Their Sobolev regularities are given by
	\[
	\frac d2-\frac d{r_\gamma}-\frac1{a_\gamma}
	=
	\frac{2d\gamma}{d+1}
	-\frac{(d-1)\gamma}{d+1}
	=\gamma.
	\]
	
	Since \(\alpha+\beta=1\), we also have
	\[
	\frac1{a_\alpha}+\frac1{a_\beta}
	=\frac{\alpha+\beta}{p}
	=\frac1p,
	\]
	and
	\[
	\frac1{r_\alpha}+\frac1{r_\beta}=
	1-\frac{2(\alpha+\beta)}{d+1}=\frac{d-1}{d+1}=\frac1p.
	\]
	This proves \eqref{eq:holder}.
	
	Finally,
	\[
	\begin{aligned}
		s-\alpha-\frac12=\frac d{d+1}-1+\frac d{d^2-1}=-\frac1{d+1}+\frac d{d^2-1}=\frac1{d^2-1} =\delta.
	\end{aligned}
	\]
	Using \(\alpha+\beta=1\) and \(1+1/p=2s\), we obtain
	\[
	\beta+\frac1p-\frac12
	=
	2s-\left(\alpha+\frac12\right)
	=
	s+\delta.
	\]
	This proves \eqref{eq:gain-identities}.
\end{proof}

\subsection{Radial decomposition of the higher annulus}

\begin{proposition}\label{prop:bilinear-high}
	For every $d\ge3$, the estimate in Proposition~\ref{prop:bilinear-intro} holds with \(\delta_d=1/(d^2-1)\).
\end{proposition}
\begin{proof}
	Fix $d\ge3$, and let $p$, $s$, $\alpha$, $\beta$, $\delta$ be as in Lemma \ref{lem:parameters}. Partition $[N/2, 2N]$ into consecutive half-open intervals $I_\ell$, each of length $M$, except possibly the last. Since $M\le N$, their number satisfies $K\lesssim \frac NM$.
	For each $\ell$, define $f_{N,\ell}$ by
	\begin{equation}\label{eq:shell-definition}
		\widehat f_{N,\ell}(\xi)
		=\mathbf1_{I_\ell}(|\xi|)\widehat f_N(\xi),
	\end{equation}
	where \(\mathbf{1}_{I_\ell}(\cdot)\) is the characteristic function of \(I_\ell\).
	The disjointness of the intervals and Plancherel's theorem give
	\begin{equation}\label{eq:shell-orthogonality}
		f_N=\sum_{\ell=1}^Kf_{N,\ell},
\qquad\sum_{\ell=1}^K\norm{f_{N,\ell}}_{L^2}^2=\norm{f_N}_{L^2}^2.
	\end{equation}
	Set $u_\ell(x,t)=\U(t)f_{N,\ell}(x)$ and $v(x,t)=\U(t)g_M(x)$. Since $\operatorname{supp}\widehat f_{N,\ell}\subset\mathcal \A_N$, Lemma \ref{lem:strichartz} applies to each $u_\ell$ with a constant independent of $\ell$, $N$ and $M$. H\"older's inequality in space and time, together with \eqref{eq:holder} and Lemma \ref{lem:strichartz}, therefore yields
	\begin{equation}
		\label{ulv 1}
		\norm{u_{\ell} v}_{L^p_{x,t}}\le \norm{u_{\ell}}_{L_t^{a_\alpha}L_x^{r_\alpha}}\norm{v}_{L_t^{a_\beta}L_x^{r_\beta}} \le C_dN^\alpha M^\beta  \norm{f_{N,\ell}}_{L^2}\norm{g_M}_{L^2}.
	\end{equation}
	Let $\rho_\ell$ be the left point of $I_\ell$, so that $\overline{I_\ell}\subset [\rho_\ell,\rho_\ell+M]$. Since $\operatorname{supp}\widehat g_M\subset \A_M$, Lemma \ref{lem:temporal-support} gives, for every $x\in \mathbb R^d$,
	$$\operatorname{supp}\mathcal F_t\big(u_\ell(x,\cdot)v(x,\cdot)\big)
	\subset[\rho_\ell+M/2,\rho_\ell+3M].$$
	In particular, this support is contained in $[\rho_\ell-3M,\rho_\ell+3M]$. By Lemma \ref{lem:temporal-support}, the product $u_\ell v$ is bounded and jointly continuous. Together with \eqref{ulv 1}, these properties allow us to apply Lemma \ref{lem:time-bernstein} with $A=\rho_{\ell}$ and $L=3M$. We obtain
	\begin{equation}\label{eq:shell-maximal}
		\left\|\sup_{t\in\mathbb R}|u_\ell v|\right\|_{L_x^p}\le C_pM^{1/p}\|u_\ell v\|_{L_{x,t}^p}
		\le C_{d,p}N^\alpha M^{\beta+1/p}
		\norm{f_{N,\ell}}_{L^2}\norm{g_M}_{L^2}.
	\end{equation}
	The Fourier integral representation yields the pointwise identity $U(t)f_N=\sum_{\ell=1}^Ku_\ell$. Taking the supremum over time and then applying the triangle inequality in $L_x^p$, we obtain
	\begin{equation}
		\label{eq:shell-sum}
		\left\|\sup_{t\in\mathbb R}|\U(t)f_N\,v|\right\|_{L^p}\le\sum_{\ell=1}^K
		\left\|\sup_{t\in\mathbb R}|u_\ell v|\right\|_{L^p} \le C_{d,p}N^\alpha M^{\beta+1/p}
		\norm{g_M}_{L^2}\sum_{\ell=1}^K\norm{f_{N,\ell}}_{L^2}.
	\end{equation}
	By Cauchy-Schwarz inequality, the orthogonality identity \eqref{eq:shell-orthogonality},  we have
	$$ \sum_{\ell=1}^K\norm{f_{N,\ell}}_{L^2}\leq K^{1/2}\bigg(\sum_{l=1}^K\|f_{N,\ell}\|_{L^2}^2\bigg)^{1/2}
	= K^{1/2}\norm{f_N}_{L^2}
	\lesssim (N/M)^{1/2}\norm{f_N}_{L^2}.$$
	Substituting this bound into \eqref{eq:shell-sum} yields
	\begin{equation}\label{eq:before-normalization}
		\left\|\sup_{t\in\mathbb R}|\U(t)f_N\,\U(t)g_M|\right\|_{L^p}
		\le C_dN^{\alpha+1/2}M^{\beta+1/p-1/2}
		\norm{f_N}_{L^2}\norm{g_M}_{L^2}.
	\end{equation}
	Finally, the identities in \eqref{eq:gain-identities} give
	$$N^{\alpha+1/2}M^{\beta+1/p-1/2}=N^{s-\delta}M^{s+\delta}=\bigg(\frac{M}{N}\bigg)^\delta N^sM^s.$$
	This proves the desired estimate with $\delta_d=\delta=1/(d^2-1)$.
\end{proof}
\begin{remark}\label{rem:sharp-cut}
	The characteristic functions in \eqref{eq:shell-definition} are used only to decompose the data orthogonally in \(L^2\). Each \(f_{N,\ell}\) satisfies the annular Strichartz estimate of Lemma~\ref{lem:strichartz}, the proof of which relies on a single smooth cut-off equal to \(1\) on \(\A_N\). Consequently, no \(L^r\) multiplier bound for these characteristic functions is required.
	
	The partition is chosen separately for each pair \((N,M)\). The subsequent summation over frequencies relies only on the resulting estimate, whose constant is uniform in both scales. Finally, the decomposition restricts only the radial frequency variable and imposes no radial symmetry on the data.
\end{remark}

\subsection{The two-dimensional estimate}

\begin{proposition}\label{prop:bilinear-two}
	Proposition~\ref{prop:bilinear-intro} holds when \(d=2\),
	with \(p_2=3\), \(s_2=2/3\), and \(\delta_2=1/30\).
\end{proposition}

\begin{proof}
	In order to apply Lemma \ref{lem:strichartz}, we choose the two finite, non-endpoint pairs
	\[
	(a_H,r_H)=(10,10/3),\qquad
	(a_L,r_L)=(30/7,30).
	\]
	They satisfy \(1/a_H+1/(2r_H)=1/a_L+1/(2r_L)=1/4\).
	The corresponding regularities are
	\[
	\gamma_{H}=1-\frac{ 2}{r_{H}}-\frac{1}{a_{H}} = \frac3{10},
	\]
	and
	\[
	\gamma_{L}=1-\frac{ 2}{r_{L}}-\frac{1}{a_{L}} =\frac7{10}.
	\]
	Moreover,
	\[\frac{1}{a_{H}} + \frac{1}{a_{L}} = \frac{1}{r_{H}} + \frac{1}{r_{L}} =\frac{1}{3}= \frac{1}{p_{2}}. \]
	H\"older's inequality   and Lemma~\ref{lem:strichartz}  therefore give
	\begin{equation}\label{Strichartzd=2}
		\norm{\U(t)f_N\,\U(t)g_M}_{L^3_{x,t}}
		\le  \|U(t)f_{N}\|_{L^{a_{H}}_{t}L^{r_{H}}_{x}} \|U(t)g_{M}\|_{L^{a_{L}}_{t}L^{r_{L}}_{x}} \le C N^{3/10}M^{7/10}\norm{f_N}_{L^2}\norm{g_M}_{L^2}.
	\end{equation}
	Lemma~\ref{lem:temporal-support} places the temporal Fourier support with respect to $t$
	of the product $U(t)f_{N}U(t)g_{M}$ in an interval contained in \([0,4N]\).
	Indeed, the support is contained in
	\([N/2+M/2,2N+2M]\), and \(M\le N\).
	After modulation by \(e^{-2iNt}\), this interval lies inside
	\([-2N,2N]\), so Lemma~\ref{lem:time-bernstein} applies with
	\(L=2N\). Lemma~\ref{lem:time-bernstein} and (\ref{Strichartzd=2}) yield
	\[
	\left\|\sup_{t\in\mathbb R}|\U(t)f_N\,\U(t)g_M|\right\|_{L_{x}^{3}}
	\le C N^{3/10+1/3}M^{7/10}\norm{f_N}_{L^2}\norm{g_M}_{L^2}.
	\]
	Finally,
	\[
	\frac3{10}+\frac13-s_{2} = \frac3{10}+\frac13-\frac23=-\frac1{30}=-\delta_{2},
	\]
	and
	\[
	\frac7{10} -s_{2}=\frac7{10}-\frac23=\frac1{30} = \delta_{2}.
	\]
	This proves   Proposition~\ref{prop:bilinear-intro}.
\end{proof}

Propositions~\ref{prop:bilinear-high} and~\ref{prop:bilinear-two}
prove Proposition~\ref{prop:bilinear-intro}. In Appendix  \ref{app:parameters},  we will further  explain how theese exponents
$a_{H},r_{H},a_{L},r_{L}$, and $\delta_{d}$ were chosen and how radial subdivision yields a unified argument for all $d\ge3$.

\section{Frequency summation and the maximal estimate}
\label{sec:main-proof}

\subsection{A finite-frequency square expansion}

\begin{lemma}\label{lem:finite-sum}
	Let \(S\subset\mathbb Z\) be finite, $q_d$ and $s_d$ be defined by (\ref{eq:critical-pair}). Suppose
	\(h_j\in L^2(\R^d)\) with
	\(\supp\widehat h_j\subset\A_{2^j}\).
	Then
	\begin{equation}\label{eq:finite-sum}
		\left\|\sup_{t\in\mathbb R}\left|\sum_{j\in S}\U(t)h_j\right|\right\|_{L^{q_d}}^2
		\le C_d\sum_{j\in S}2^{2s_dj}\norm{h_j}_{L^2}^2.
	\end{equation}
\end{lemma}

\begin{proof}
	Let \(u_j=\U(t)h_j\), \(a_j=2^{s_dj}\norm{h_j}_{L^2}\),
	and extend \(a_j\) by zero outside \(S\).
	For a complex number \(z\), \(|z^2|=|z|^2\).
	Since   \(q_d=2p_d\), which can be find in (\ref{eq:delta}), it follows that
	\begin{align}
		\left\|\sup_{t\in\mathbb R}\left|\sum_j u_j\right|\right\|_{L^{q_d}}^2
		&=\left\|\sup_{t\in\mathbb R}\left|\left(\sum_j u_j\right)^2\right|
		\right\|_{L^{p_d}} \notag\\
		&\le \sum_{j,k\in S}\left\|\sup_t|u_ju_k|\right\|_{L^{p_d}}.
		\label{eq:square}
	\end{align}
	We expanded the complex square instead of introducing a conjugate factor, so all product estimates are for the same half-wave sign.
	
	Proposition~\ref{prop:bilinear-intro}, with the larger of the two
	frequencies in the first position, gives
	\[
	\left\|\sup_{t\in\mathbb R}|u_ju_k|\right\|_{L^{p_d}}
	\le C_d2^{-\delta_d|j-k|}a_ja_k.
	\]
	Summing and applying Cauchy--Schwarz inequality to each shift, we obtain
	\begin{align}
		\sum_{j,k}2^{-\delta_d|j-k|}a_ja_k
		&=\sum_{m\in\mathbb Z}2^{-\delta_d|m|}
		\sum_{j\in\mathbb Z}a_ja_{j-m}\notag\\
		&\le
		\left(\sum_{m\in\mathbb Z}2^{-\delta_d|m|}\right)
		\sum_j a_j^2 \nonumber\\
		&= \frac{1+2^{-\delta_d}}{1-2^{-\delta_d}}  \sum_{j\in S}2^{2s_dj}\norm{h_j}_{L^2}^2.
		\label{eq:discrete}
	\end{align}
	The constant
	\(\frac{1+2^{-\delta_d}}{1-2^{-\delta_d}}\)  is finite for
	every fixed \(d\).
	Equations \eqref{eq:discrete} and \eqref{eq:square} imply \eqref{eq:finite-sum}.
\end{proof}

\subsection{Passage to Schwartz and Sobolev data}

Choose a real-valued, non-negative, radial function
\(\varphi\in C_c^\infty(\R^d\setminus\{0\})\), with support contained in
\(\{\xi \in \mathbb{R}^{d}: 1/2\le|\xi|\le2\}\), such that
\[
\sum_{j\in\mathbb Z}\varphi(2^{-j}\xi)=1
\quad(\xi\ne0).
\]

For each Schwartz function $f$ defined on $\mathbb{R}^{d}$, we denote
\[ \widehat{P_{j}f}(\xi): =\varphi(2^{-j}\xi) \hat{f}(\xi). \]
Plancherel's theorem, the fact that  \(|\xi| \sim 2^j\) on each support,
and the partition of unity yield
\begin{equation}\label{eq:LP}
	\sum_j2^{2s_dj}\norm{P_jf}_{L^2}^2
	\sim \norm{f}_{\dot H^{s_d}}^2.
\end{equation}
%For completeness, the upper bound follows by finite overlap.
%For the lower bound, if at most \(B\) terms are nonzero at a fixed
%\(\xi\), then
%\[
% 1=\left|\sum_j\varphi(2^{-j}\xi)\right|^2
%       \le B\sum_j|\varphi(2^{-j}\xi)|^2.
%\]
%Integrating with the weight \(|\xi|^{2s_d}|\widehat f(\xi)|^2\)
%proves the reverse inequality.

Let
\[
f^{(J)}=\sum_{|j|\le J}P_jf.
\]
By Lemma~\ref{lem:finite-sum} and \eqref{eq:LP},
\begin{equation}\label{eq:uniform-truncation}
	\norm{\M_{\R}f^{(J)}}_{L^{q_d}}
	\le C_d\norm{f}_{\dot H^{s_d}},
\end{equation}
uniformly in \(J\).
For \(f\in\mathcal S(\R^d)\), the finite overlap property of the supports of functions $\widehat{P_{j}f}$  also implies
\[
\sum_j\norm{\widehat{P_jf}}_{L^1}\le C\norm{\widehat f}_{L^1}<\infty.
\]
Thus
\begin{equation}\label{eq:uniform-xt}
	\sup_{x,t}\left|\U(t)(f-f^{(J)})(x)\right|
	\le (2\pi)^{-d}\norm{\widehat f-\widehat{f^{(J)}}}_{L^1}
	\longrightarrow0.
\end{equation}
In particular, the corresponding maximal functions converge
pointwisely. Applying Fatou's Lemma to \eqref{eq:uniform-truncation} gives
\[  \norm{\M_{\R}f}_{L^{q_d}} =  \norm{ \lim_{J \rightarrow +\infty}  \M_{\R}f^{(J)}}_{L^{q_d}} \le \liminf_{J \rightarrow +\infty} \norm{\M_{\R}f^{(J)}}_{L^{q_d}}
\le C_d\norm{f}_{\dot H^{s_d}}. \]
This proves Theorem~\ref{thm:main} for Schwartz function.
Notice that the Schwartz class \(\mathcal{S}(\mathbb{R}^d)\) is dense in
\(\dot H^{s_d}(\mathbb{R}^d)\) with respect to the
\(\dot H^{s_d}\)-norm. The extension to
\(\dot H^{s_d}(\mathbb{R}^d)\) follows from density and the bounded
linear extension theorem \cite[Theorem~I.7]{RS}: apply the theorem to
the evolution operator with values in
\(L^{q_d}(\mathbb{R}^d;C([-T,T]))\), and then let \(T\to\infty\).
See also \cite[the proof of Theorem~5]{BG} for a similar
maximal-norm approximation argument. We omit the details of this
standard argument.

\appendix

\section{Derivation of the parameters and the dimension split}
\label{app:parameters}

We explain how the exponents in the proof were selected, and
why the radial subdivision removes a restriction on dimension.
We fix \(d\ge2\), and write \(p=(d+1)/(d-1)\), \(s=d/(d+1)\).
On the sharp wave-admissibility line in Lemma~\ref{lem:strichartz}, the exponents $a$ and $r$ satisfy
\[
\frac1a+\frac{d-1}{2r}=\frac{d-1}{4},
\]
and the Strichartz regularity  satisfies
\[
\gamma=\frac d2-\frac dr-\frac1a
=\frac{d+1}{d-1}\frac1a=\frac pa.
\]
Solving for both exponents gives
\begin{equation}\label{eq:derived-pairs}
	\frac1a=\frac\gamma p,\qquad
	\frac1r=\frac12-\frac{2\gamma}{d+1}.
\end{equation}
Thus a finite, non-endpoint pair with \(a,r>2\) is obtained
precisely when
\begin{equation}\label{eq:gamma-range}
	0<\gamma<\min\left\{\frac p2,\frac{d+1}{4}\right\}.
\end{equation}
For two regularities \(\alpha,\beta\) satisfying
\(\alpha+\beta=1\), \eqref{eq:derived-pairs} gives
\[
\frac1{a_\alpha}+\frac1{a_\beta}=\frac1p,
\qquad
\frac1{r_\alpha}+\frac1{r_\beta}
=1-\frac2{d+1}=\frac1p.
\]
These are exactly the two product H\"older identities.

\subsection{Why the unsliced argument is restricted to low dimensions}

Applying temporal Bernstein directly at scale \(N\) gives
\[
\left\|\sup_t|\U(t)f_N\U(t)g_M|\right\|_{L^p}
\le C N^{\alpha+1/p}M^\beta\norm{f_N}_{L^2}\norm{g_M}_{L^2}.
\]
Its normalized frequency gain is
\[
\delta_{\mathrm{unsliced}}
=s-\frac1p-\alpha=\frac1{d+1}-\alpha.
\]
Thus positive gain requires \(\alpha<1/(d+1)\), whereas
\(a_\beta =\frac{p}{1-\alpha}>2\) requires
\[
\alpha>1-\frac p2=\frac{d-3}{2(d-1)}.
\]
For \(d\ge3\), these strict inequalities can hold simultaneously
exactly when
\[
\frac{d-3}{2(d-1)}<\frac1{d+1}
\quad\Longleftrightarrow\quad d^2-4d-1<0.
\]
For integer \(d\ge3\), this means \(d=3,4\).
The spatial-exponent restrictions in \eqref{eq:gamma-range}
can also be satisfied in these dimensions.
For example, \(\alpha=1/8\) works in dimension three,
and \(\alpha=11/60\) works in dimension four.
The corresponding unsliced gains are \(1/8\) and \(1/60\).
For \(d\ge5\), this particular unsliced non-endpoint
H\"older--Bernstein argument has no admissible choice with
positive gain.

\subsection{The radial subdivision in every dimension at least three}

For \(d\ge3\), one has \(1<p\le2\).
Seek \(0<\alpha<\beta=1-\alpha<1\).
The stricter of the two time conditions is
\(\beta<p/2\), or \(\alpha>1-p/2\).
After the radial subdivision, the frequency coefficient is
\[
N^{\alpha+1/2}M^{\beta+1/p-1/2}
=N^{s-\delta}M^{s+\delta},\qquad
\delta=\frac1{2p}-\alpha,
\]
because \(2s=1+1/p\).
Positive gain and a nonendpoint lower time exponent therefore
require exactly
\begin{equation}\label{eq:selection-interval}
	1-\frac p2<\alpha<\frac1{2p}.
\end{equation}
This interval has length
\[
\frac1{2p}-1+\frac p2
=\frac{(p-1)^2}{2p}
=\frac2{d^2-1}>0.
\]
The midpoint of the interval $(1-\frac{p}{2}, \frac{1}{2p})$ is
\[
\alpha=\frac12\left(1-\frac p2+\frac1{2p}\right)
=\frac12-\frac d{d^2-1}.
\]
This gives exactly
\[
\beta=\frac12+\frac d{d^2-1},\qquad
\delta=\frac1{d^2-1}.
\]
Both spatial exponents are finite: \(0<\alpha<\beta<1\) and
\(1\le(d+1)/4\) imply the spatial part of
\eqref{eq:gamma-range}. The high time condition $\alpha <\frac{2}{p}$ follows as well,
since \(\alpha<1/(2p)<p/2\).
Thus the allowed interval remains nonempty in every fixed
dimension, even though its length tends to zero.

For a direct comparison, the ratio of the subdivided coefficient
to the unsubdivided coefficient is
\[
\frac{N^{\alpha+1/2}M^{\beta+1/p-1/2}}
{N^{\alpha+1/p}M^\beta}
=\left(\frac MN\right)^{1/p-1/2}.
\]
It improves the bound when \(d>3\), is neutral when \(d=3\),
and would worsen it when \(d=2\).

\subsection{The two-dimensional interior choice}

For \(d=2\), one has \(p=3\) and \(s=2/3\).
The spatial conditions in \eqref{eq:gamma-range} require
\(\alpha,\beta<3/4\), so \(\beta=1-\alpha\) forces
\(\alpha>1/4\). Positive unsliced gain requires
\(\alpha<1/3\). Hence the interval
\[
\frac14<\alpha<\frac13
\]
supplies strictly admissible choices. Taking \(\alpha=3/10\)
gives \(\beta=7/10\), and \eqref{eq:derived-pairs} gives
\[
a_\alpha=\frac3{3/10}=10,\qquad
\frac1{r_\alpha}=\frac12-\frac{2(3/10)}3=\frac3{10},
\]
\[
a_\beta=\frac3{7/10}=\frac{30}7,\qquad
\frac1{r_\beta}=\frac12-\frac{2(7/10)}3=\frac1{30}.
\]
Finally, we have \(\delta=1/3-3/10=1/30\).
The alternative radial condition
\(\alpha<1/(2p)=1/6\) is incompatible with the condition
\(\alpha>1/4\), so the direct argument is the appropriate
choice in this dimension.

\begin{flushleft}
	\vspace{0.3cm}\textsc{Bowen An\\
	School of Mathematics and Statistics, Northwestern Polytechnical University, 710129, Xi'an, People's Republic of China\\
		Email address:}  anbowen@mail.nwpu.edu.cn

	\vspace{0.3cm}\textsc{Wenjuan Li\\
		School of Mathematics and Statistics, Northwestern Polytechnical University, 710129, Xi'an, People's Republic of China\\
		Email address:} liwj@nwpu.edu.cn

	\vspace{0.3cm}\textsc{Huiju Wang (Corresponding Author)\\
		School of Mathematics and Statistics, Henan University, 475000, Kaifeng, People's Republic of China\\
		Email address:} huijuwang@mail.nwpu.edu.cn
\end{flushleft}

\end{document}